\documentclass{ws-m3as}
\usepackage[super]{cite}

\newtheorem{algorithm}[theorem]{Algorithm}
\begin{document}

\markboth{Y. Li \& J. W. Siegel}{Entropy-based convergence rates of greedy algorithms}

%
\catchline{}{}{}{}{}
%

\title{Entropy-based convergence rates of greedy algorithms}

\author{Yuwen Li\footnote{Corresponding author.}}

\address{School of Mathematical Sciences, Zhejiang University, \\
Hangzhou, Zhejiang 310058, China\\
liyuwen@zju.edu.cn}

\author{Jonathan W. Siegel}

\address{Department of Mathematics, Texas A\&M University,\\
College Station, TX 77845, USA\\
jwsiegel@tamu.edu}

\maketitle

\begin{history}
\received{29 April 2023}
\revised{13 November 2023}
\accepted{30 December 2023}
\end{history}

\begin{abstract}
We present convergence estimates of two types of greedy algorithms in terms of the {entropy numbers} of underlying compact sets. In the first part, we measure the error of a standard greedy reduced basis method for parametric PDEs by the entropy numbers of the solution manifold in Banach spaces. This contrasts with the classical analysis based on the Kolmogorov $n$-widths and enables us to obtain direct comparisons between the algorithm error and the entropy numbers, where the multiplicative constants are explicit and simple. The entropy-based convergence estimate is sharp and improves upon the classical width-based analysis of reduced basis methods for elliptic model problems. In the second part, we derive a novel and simple convergence analysis of the classical orthogonal greedy algorithm for nonlinear dictionary approximation using the entropy numbers of the symmetric convex hull of the dictionary. This also improves upon existing results by giving a direct comparison between the algorithm error and the entropy numbers.
\end{abstract}



\maketitle
\section{Introduction}
Greedy algorithms are ubiquitous in advanced scientific computing and computational mathematics. Successful greedy-type numerical algorithms include adaptive finite element \cite{BabuskaRheinboldt1978,CKNS2008,MorinSiebertVeeser2008,Li2019SINUM,Li2021MCOM} and wavelet  \cite{CohenDahmenDeVore2001,Gantumur2007} methods for boundary value problems, certified model reduction technique for parametric differential equations \cite{BarraultMadayNguyen2004,HesthavenRozzaStamm2016,Quarteroni2016}, and nonlinear dictionary approximation \cite{Jones1992,DeVoreTemlyakov1996,Friedman2001,BarronCohenDahmenDeVore2008,Temlyakov2008,JonathanXu2022} in signal processing, machine learning and statistics. The aforementioned algorithms use local extreme criteria and iterations to search a global quasi-optimizer. It is known that many greedy algorithms have guaranteed error controls, solid convergence analysis, and optimal computational complexity as advantages. In this paper, we are focused on two applications of greedy algorithms, one for the \emph{Reduced Basis Method} (RBM) and one for approximating functions. We shall show that these two types of greedy algorithms originally designed for seemingly unrelated problems can be analyzed in a unified way.

The RBM is a class of popular numerical methods developed in recent decades for efficiently solving  \emph{parametric} PDEs \cite{Maday2006,CohenDeVore2015,HesthavenRozzaStamm2016,Quarteroni2016}.  In a Banach space $X$, let $\mathcal{P}(u_\mu,\mu)=0$ be the model problem with the solution $u_\mu\in X$, where  $\mu$ is a parameter in a compact set $\mathcal{U}$. A standard RBM is split into two stages. In the \emph{offline} stage, the RBM builds a highly accurate finite-dimensional subspace $X_n={\rm span}\{u_{\mu_1},\ldots,u_{\mu_n}\}\subset X$ to uniformly approximate the compact solution manifold \[
\mathcal{M}=u(\mathcal{U})=\{u_\mu\in X: \mu\in\mathcal{U}\}.
\]
The offline construction of $X_n$ is potentially very expensive but is only implemented once during the offline setup stage of the RBM. Then the  \emph{online} module of the RBM is able to rapidly produce  accurate Galerkin approximations to  $u_\mu$ based on $X_n$ for many instances of the parameter $\mu$. 

There are two main types of computational methods for constructing $X_n$. One is called the Proper Orthogonal Decomposition (POD) \cite{BennerGugercinWillcox2015}, which utilizes a (costly) singular value decomposition to extract the low-rank structure of a given high-quality dataset. The other, is an iterative greedy algorithm for building the RBM subspace $X_n$, which relies on a posteriori error estimates. Compared with POD, greedy algorithms are computationally more feasible and have certified error bounds and convergence rates. The work \cite{MadayPateraTurinici2002} provided a priori convergence rates of greedy RBMs for single-parameter problems. In general, the convergence of greedy algorithms for the RBM is given by comparing the numerical error of the algorithm $\sigma_n(\mathcal{M})$ against the Kolmogorov $n$-width $d_n(\mathcal{M})$, which measures the performance of the best possible $n$-dimensional subspace for approximating the solution manifold $\mathcal{M}$. A direct comparison between $\sigma_n(\mathcal{M})$ and $d_n(\mathcal{M})$ was developed in Ref.~\citen{BuffaMaday2012}. Later on, the works \cite{BinevCohenDahmenDeVore2011,DeVorePetrova2013} showed for the first time that 
\begin{equation}\label{rateRBMthm}
    d_n(\mathcal{M})=\mathcal{O}(n^{-s})\Longrightarrow\sigma_n(\mathcal{M})=\mathcal{O}(n^{-s}),
\end{equation}
where $s>0$ and $\mathcal{O}$ is Landau's big O notation. In other words, greedy-type RBMs are rate-optimal for polynomial decaying $d_n(\mathcal{M})$.

As far as we know, convergence results for the greedy RBM in the literature are all based on comparison with the Kolmogorov $n$-width. In this paper, we derive convergence analysis of greedy-type RBMs in terms of $\varepsilon_n=\varepsilon_n({\rm co}(\mathcal{M}))$, the entropy number \cite{KolmogorovTihomirov1961} of the symmetric convex hull ${\rm co}(\mathcal{M})$. Just as the $n$-width $d_n=d_n(\mathcal{M})$, $\varepsilon_n$ is also an asymptotically small quantity for compact sets (see Refs.~\citen{Carl1981} and \citen{Lorentz1996}).
Unlike the rate-optimality result \eqref{rateRBMthm}, our analysis leads to  direct and optimal comparisons of the type 
\begin{equation}\label{sigmaepsilon}
\sigma_n(\mathcal{M})\leq Cn^t\varepsilon_n({\rm co}(\mathcal{M})),\quad t\in \left[\frac{1}{2},1\right],   
\end{equation}
where $C$ is an explicit constant and the exponent $t$ depends upon the Banach space in which the error is measured. Specifically, $t = (1/2) + |1/2 - 1/p|$ for the spaces $L_p$. As $n\to\infty$  \eqref{sigmaepsilon} is sharper than the direct comparisons based on $d_n(\mathcal{M})$ given in Refs.~\citen{BuffaMaday2012} and \citen{DeVorePetrova2013}. Moreover, in the context of solving parametric PDEs, it is easier to calculate the entropy numbers than the Kolmogorov $n$-width of relevant sets, which is another  advantage of bounding $\sigma_n(\mathcal{M})$ in terms of $\varepsilon_n$, see Subsection \ref{subsecExamples} for more details.

On the other hand, there are a variety of greedy algorithms for constructing nonlinear approximations (see Ref.~\citen{DeVore1998}) of a single target function in the fields of machine learning, statistics and signal processing. The historically first of these algorithms is known as projection pursuit regression in statistics \cite{friedman1981projection,jones1987conjecture}, the matching pursuit algorithm in signal procession \cite{MallatZhang1993,DavisMallat1997}, and the Pure Greedy Algorithm (PGA) in approximation theory \cite{DeVoreTemlyakov1996}. Other variants of these greedy algorithms have been proposed which enjoy significantly better convergence guarantees for general dictionaries, including the Relaxed Greedy Algorithm (RGA) \cite{Jones1992,DeVoreTemlyakov1996,petrova2016rescaled,DereventsovTemlyakov2019} and the Orthogonal Greedy Algorithm (OGA) \cite{Pati1993,DeVoreTemlyakov1996,BarronCohenDahmenDeVore2008}. 
We refer to the article \cite{Temlyakov2008} and the book \cite{Temlyakov2011} for a thorough introduction to greedy algorithms for function approximation. 

In recent decades, greedy algorithms have also been used to solve PDEs, see \cite{figueroa2012greedy,le2009results,ammar2006new} and the recent work \cite{Siegel2023} which uses RGAs and OGAs to solve PDEs with neural networks. These algorithms adaptively select basis functions from a redundant dictionary $\mathbb{D}\subset X$ and use a sparse linear combination $$f_n=\sum_{i=1}^nc_ig_i,\quad g_i\in\mathbb{D}$$ of dictionary elements to approximate a target function $f$. Among these greedy algorithms, the RGA and OGA achieve the worst case optimal convergence rate $\mathcal{O}(n^{-1/2})$ for target functions in the convex hull of the dictionary $\mathbb{D}$ \cite{DeVoreTemlyakov1996,petrova2016rescaled}, while the PGA in general may perform worse \cite{livshits2009lower}. In Ref.~\citen{JonathanXu2022}, it is shown that for dictionaries whose convex hull has small entropy numbers, the orthogonal greedy algorithm may converge faster, specifically for $f\in{\rm co}(\mathbb{D})$ and $s>1/2$, 
\begin{equation}\label{ffnepsilonrate}
   \varepsilon_n({{\rm co}(\mathbb{D})})=\mathcal{O}(n^{-s})\Longrightarrow\|f-f_n\|_X=\mathcal{O}(n^{-s}).
\end{equation}
As mentioned before, the OGA converges with rate $\mathcal{O}(n^{-1/2})$ and no better for a general dictionary $\mathbb{D}$ (see Ref.~\citen{DeVoreTemlyakov1996}). The result \eqref{ffnepsilonrate} is an improved convergence rate for OGAs under additional assumptions, which hold for many popular dictionaries, e.g., the $\texttt{ReLU}_k$ dictionary in shallow neural networks (see Ref.~\citen{JonathanXuFoCM}). 

The second main contribution of our work is a novel direct comparison for the OGA:
\begin{equation}\label{ffnepsilon0}
    \|f-f_n\|_X\leq C_f\varepsilon_n({\rm co}(\mathbb{D})),
\end{equation}
where $C_f$ is an explicit uniform constant depending on the target function $f$. We remark that \eqref{ffnepsilonrate} is indeed a \emph{rate} comparison while our estimate  \eqref{ffnepsilon0} is \emph{direct} and implies \eqref{ffnepsilonrate} as special cases. In addition, the validity of \eqref{ffnepsilonrate} relies on the polynomial decay of  $\varepsilon_n({\rm co}(\mathbb{D}))$, while the estimate \eqref{ffnepsilon0} is non-asymptotic and is always valid for general dictionaries, e.g., when $\varepsilon_n({\rm co}(\mathbb{D}))=\mathcal{O}(e^{-\alpha n^s})$ is exponentially diminishing.
  
Throughout the rest of this section, we briefly introduce the Kolmogorov $n$-width, the entropy numbers, and the details of greedy algorithms for the RBM and function approximation.

\subsection{Greedy Algorithms for RBMs}
Let $K$ be a compact set in a Banach space $X$ equipped with the norm $\|\bullet\|=\|\bullet\|_X$. The Kolmogorov $n$-width of $K$ is defined as 
\begin{equation*}
d_n(K)=d_n(K)_X:=\inf_{X_n}\sup_{f\in K}\text{dist}(f,X_n),
\end{equation*}
where $${\rm dist}(f,X_n):=\inf_{g\in X_n}\|f-g\|$$ is the distance from $f$ to $X_n$, and the infimum is taken over all $n$-dimensional subspaces of $X$. 

On the other hand, {Kolmogorov proposed and investigated} the (dyadic) entropy numbers of $K$ (see Ref.~\citen{KolmogorovTihomirov1961}), which are given by \[
\varepsilon_n(K)=\varepsilon_n(K)_X:=\inf\{\varepsilon>0: K\text{ is covered by }2^n\text{ balls of radius } \varepsilon\}.
\]
It is known that $\{d_n(K)\}_{n\geq1}$ and $\{\varepsilon_n(K)\}_{n\geq1}$ are decreasing sequences and $\lim_{n\to\infty}d_n(K)=\lim_{n\to\infty}\varepsilon_n(K)=0$ for any compact set $K$ (see Ref.~\citen{Lorentz1996}). When $K$ is compact, the symmetric or absolutely convex hull of $K$,
\[
\text{co}(K):=\overline{\Big\{\sum_{i}c_ig_i:   \sum_{i}|c_i|\leq1,~g_i\in K\text{ for each }i\Big\}},
\]
is also compact.
The key quantity used in our convergence analysis is the entropy number $\varepsilon_n({\rm co}(K))$, while in traditional analyses the Kolmogorov $n$-widths $d_n(K)$ are typically used.

In the offline stage, the  greedy algorithm for RBMs selects a sequence $\{f_n\}_{n\geq1}$  from $K$ in an adaptive way as follows.
\begin{algorithm}[Greedy Algorithm for RBMs]\label{GA}
Set $n=1$ and $X_0=\{0\}$. 
\begin{itemize}
\item[]\textsf{Step 1}: Compute $$f_n=\arg\max_{f\in K}{\rm dist}(f,X_{n-1}).$$
\item[]\textsf{Step 2}: Set $X_n={\rm span}\{f_1,\ldots,f_n\}$. Set $n=n+1$. Go to \textsf{Step 1}.
\end{itemize}
\end{algorithm}
In practice, $K$ is often the solution manifold $\mathcal{M}$ of a parametric PDE corresponding to a range of varying parameters, and 
the computational cost of the exact  greedy approach (Algorithm \ref{GA}) might still be exceedingly high. In such cases, a weak greedy algorithm driven by a posterior error estimation is able to relax the \verb|argmax| criterion in Algorithm \ref{GA} and to realize the next more feasible algorithm. This modification is analogous to the weak greedy algorithms for function approximation \cite{Temlyakov2000}, which we discuss in more detail in  Subsection \ref{subsecOGA}.
\begin{algorithm}[Weak Greedy Algorithm for RBMs]\label{WGA}
Set $n=1$, $\gamma\in(0,1]$ and $X_0=\{0\}$. 
\begin{itemize}
\item[]\textsf{Step 1}: Select $f_n\in K$ such that $${\rm dist}(f_n,X_{n-1})\geq\gamma\max_{f\in K}{\rm dist}(f,X_{n-1}).$$
\item[]\textsf{Step 2}: Set $X_n={\rm span}\{f_1,\ldots,f_n\}$. Set $n=n+1$. Go to \textsf{Step 1}.
\end{itemize}
\end{algorithm}
We remark that \textsf{Step 1} of Algorithm \ref{WGA} is not directly implemented but practically achieved by an a posteriori error estimation \cite{RozzaHuynhPatera2008,HesthavenRozzaStamm2016,EdelMaday2023}. The constant
$\gamma\in(0,1]$ actually depends on the quality of a posteriori error estimators. The numerical error of Algorithm \ref{WGA} is measured by 
\[
\sigma_n(K)=\sigma_n(K)_X:=\sup_{f\in K}{\rm dist}(f,X_n).
\]  
By definition the simple bound \begin{equation}\label{dnsigman}
d_n(K)\leq\sigma_n(K)    
\end{equation}
is always true. We note that Algorithm \ref{WGA} with $\gamma=1$ reduces to Algorithm \ref{GA} and the iterates $\{f_n\}_{n\geq1}$ are not uniquely determined by the above greedy selection processes.
When $X$ is a Hilbert space, let $P_U$ be the orthogonal projection onto a subspace $U\subset X$. In this case we have
\begin{align*}
   &{\rm dist}(f,{X_{n}})=\|f-P_{X_{n}}f\|,\\
   &\sigma_n(K)=\max_{f\in K}\|f-P_{X_{n}}f\|. 
\end{align*}

\subsection{Greedy Algorithms for Approximating Functions}\label{subsecOGA}
Let the dictionary $\mathbb{D}$ be a collection of elements in a real Hilbert space $X$ equipped with inner product $\langle\cdot,\cdot\rangle$. For learning algorithms based upon wavelets, shallow neural networks, statistical regression or compressed sensing, it is often essential to construct nonlinear approximants to 
a target function $f\in X$. To achieve this goal, Refs.~\citen{friedman1981projection} and \citen{MallatZhang1993} developed the projection pursuit or matching pursuit algorithm: 
    \begin{align*}
&g_n=\arg\max_{g\in\mathbb{D}}|\langle g,f-f_{n-1}\rangle|,\\ &f_n=f_{n-1}+\langle f-f_{n-1},g_n\rangle g_n, 
    \end{align*}
where $f_0=0$ and $n=1,2,\ldots$ is the iteration counter. Note that if the dictionary elements are normalized in $X$, then this algorithm is adding the single term which minimizes the error in each step and validates the name ``greedy algorithm." The theoretical analysis of this algorithm is notoriously complex \cite{jones1987conjecture,DeVoreTemlyakov1996,sil2004rate}, and in any case it does not achieve an optimal rate of convergence \cite{livshits2009lower}. As an alternative, a variety of different greedy algorithms which achieve better convergence behavior have been introduced, for example the following RGA \cite{Jones1992,DeVoreTemlyakov1996,BarronCohenDahmenDeVore2008} and its variants \cite{petrova2016rescaled}:
    \begin{align*}
(\alpha_n, \beta_n, g_n)&=\arg\min_{\alpha, \beta\in\mathbb{R}, g\in\mathbb{D}}\|f-\alpha f_{n-1}-\beta g_n\|,\\ f_n&=\alpha_nf_{n-1}+\beta_ng_n.
    \end{align*}
In this work, the algorithm under consideration is the OGA which has been developed and analyzed in, e.g., Refs.~\citen{Pati1993,DeVoreTemlyakov1996,BarronCohenDahmenDeVore2008,Temlyakov2011,DereventsovTemlyakov2019} and \citen{JonathanXu2022}.
\begin{algorithm}[Orthogonal Greedy Algorithm]\label{OGA}
Set $n=1$ and $f_0=0$. 
\begin{itemize}
\item[]\textsf{Step 1}: Compute the optimizer {$$g_n=\arg\max_{g\in\mathbb{D}}|\langle g,f-f_{n-1}\rangle|.$$}
\item[]\textsf{Step 2}: Set $X_n={\rm span}\{g_1,\ldots,g_n\}$ and compute
$$f_n=P_{X_n}f=P_nf.$$ Set $n=n+1$. Go to \textsf{Step 1}.
\end{itemize}
\end{algorithm}
In \textsf{Step 1} of Algorithm \ref{OGA}, it is possible that the argmax may not exist or may be impossible to compute. To overcome this, one could relax the selection criterion to $$|\langle g_n,f-f_{n-1}\rangle|\geq\gamma\max_{g\in\mathbb{D}}|\langle g,f-f_{n-1}\rangle|$$ with $\gamma\in(0,1]$. This modified algorithm is closer to what is typically implemented in practice and is called the weak OGA \cite{Temlyakov2000}. Weak versions of other greedy algorithms have also been introduced and analyzed in e.g., Ref.~\citen{Temlyakov2011}.

In our error estimates, $C$ is a generic and uniform constant that may change from line to line but is independent of $n$ and underlying compact sets or target functions. By $A\lesssim B$ we denote $A\leq CB$, and $A\eqsim B$ is equivalent to $A\lesssim B$ and $B\lesssim A$.

The rest of this paper is organized as follows. In Section \ref{secHilbert}, we present an entropy-based convergence analysis of greedy RBMs and corresponding applications in Hilbert spaces. In Section \ref{secBanach}, we derive similar optimal convergence estimates for RBMs in Banach spaces. Section \ref{secOGA} is devoted to an optimal direct comparison estimate for the convergence of OGAs in Hilbert spaces. Concluding remarks are presented in Section \ref{secConclusion}.

\section{Convergence of RBMs in Hilbert spaces}\label{secHilbert}
Throughout this section, we assume that $X$ is a Hilbert space unless confusion arises.
The next lemma is the main tool of our entropy-based convergence analysis for greedy algorithms.
\begin{lemma}\label{mainlemma}
Let $V_n=\frac{\pi^\frac{n}{2}}{\Gamma\big(\frac{n}{2}+1\big)}$ be the volume of an $\ell_2$-unit ball in $\mathbb
R^n$.
Let $K$ be a compact set in a Hilbert space $X$. Then for any $v_1, \ldots, v_n\in K$, it holds that
\begin{equation*}
\Big(\prod_{k=1}^n\big\|v_k-P_{k-1}v_k\big\|\Big)^\frac{1}{n}\leq(n!V_n)^\frac{1}{n}\varepsilon_n({\rm co}(K))_X.
\end{equation*}
where $P_0=0$ and $P_k$ is the orthogonal projection onto $\operatorname{span}\{v_1,\ldots,v_k\}.$
\end{lemma}
\begin{proof}
Consider the symmetric simplex  \begin{equation*}
S={\rm co}\big(\{v_1,\ldots,v_n\}\big)=\left\{\sum_{i=1}^nc_iv_i: \sum_{i=1}^n|c_i|\leq1\right\}.    
\end{equation*}
We assume that $\{v_k\}_{k=1}^n$ are linearly independent otherwise the inequality automatically holds.
Let $X_n$ be the $n$-dimensional subspace spanned by $v_1,\ldots,v_n$. We identify $X_n$ with the Euclidean space $\mathbb{R}^n$ and $S$ with a $n$-dimensional skewed simplex in $\mathbb{R}^n$. It is straightforward to see that $$\{v_k\}_{k=1}^n\mapsto\{v_k-P_{k-1}v_k\}_{k=1}^n$$
is the Gram-Schmidt orthogonalization of $v_1,...,v_n$. As a result, the volume of $S$ is given by
\begin{equation}\label{svol}
    |S|=\frac{2^n}{n!}\prod_{k=1}^n\big\|v_k-P_{k-1}v_k\big\|.
\end{equation}
The definition of $\varepsilon_n=\varepsilon_n({\rm co}(K))_X$ implies that $S$ is covered by $2^n$ balls of radius $\varepsilon_n$. Therefore by \eqref{svol} and a volume comparison between $S$ and the union of balls, we have
\begin{equation*}
    \frac{2^n}{n!}\prod_{k=1}^n\big\|v_k-P_{k-1}v_k\big\|\leq 2^n\varepsilon^n_nV_n.
\end{equation*}
The proof is complete.
\end{proof}

\subsection{Convergence estimates}
Using Lemma \ref{mainlemma}, we immediately obtain a direct comparison result between the RBM error and the entropy number on convex compact sets.
\begin{theorem}\label{sigmaepsilonthm}
Let $K$ be a compact set in a Hilbert space $X$. For the weak greedy algorithm (Algorithm \ref{WGA}) we have
\begin{equation}
    \sigma_n(K)_X\leq\gamma^{-1}(n!V_n)^\frac{1}{n}\varepsilon_n({\rm co}(K))_X.
\end{equation}
\end{theorem}
\begin{proof}
For {Algorithm \ref{WGA}} and $n\geq1$, we have 
\begin{equation}\label{gammataun}
    \gamma{\sigma_{n-1}}(K)\leq\big\|f_n-P_{X_{n-1}}f_n\big\|\leq{\sigma_{n-1}}(K).
\end{equation}
It then follows from \eqref{gammataun} and Lemma  \ref{mainlemma} that
\begin{equation}\label{sigmaprod}
\begin{aligned}
    \Big(\prod_{k=1}^n{\sigma_{k-1}}(K)\Big)^\frac{1}{n}&\leq \gamma^{-1}\Big(\prod_{k=1}^n\big\|f_k-P_{X_{k-1}}f_k\big\|\Big)^\frac{1}{n}\\
    &\leq\gamma^{-1}(n!V_n)^\frac{1}{n}\varepsilon_n({\rm co}(K))_X.
\end{aligned}
\end{equation}
Combining \eqref{sigmaprod} with the fact
\[
{\sigma_0(K)\geq\sigma_1(K)\geq \cdots \geq\sigma_{n-1}(K)\geq\sigma_n(K)}
\]
completes the proof.
\end{proof}

Using the well-known Stirling's formula
\begin{equation*}
    \lim_{n\to\infty}n!\left/\sqrt{2\pi n}\Big(\frac{n}{\text{e}}\Big)^n\right.=1,\quad\lim_{n\to\infty}V_n\left/\frac{1}{\sqrt{n\pi}}\Big(\frac{2\pi\text{e}}{n}\Big)^\frac{n}{2}\right.=1,
\end{equation*}
we obtain a corollary about the convergence rate of greedy-type RBMs.
\begin{corollary}\label{sigmaepsiloncor}
    Let $K$ be a compact set in a Hilbert space $X$. For Algorithm \ref{WGA} there exists an absolute constant $C$ independent of $K$ and $n$, such that
\begin{equation*}
    \sigma_n(K)_X\leq C\sqrt{n}\varepsilon_n({\rm co}(K))_X.
\end{equation*}
\end{corollary}

A natural question is whether the convergence rate of $\sigma_n(K)$ in Corollary \ref{sigmaepsiloncor} is sharp and whether the factor $\sqrt{n}$ is removable. In the following we give a negative answer to this question.
\begin{proposition}
There exists a Hilbert space $X$ and a symmetric convex compact set $K \subset X$ such that Algorithm \ref{GA} satisfies
\begin{equation*}
    \sigma_n(K)_X\eqsim\sqrt{n}\varepsilon_n(K)_X.
\end{equation*}
\end{proposition}
\begin{proof}
On a Lipschitz domain $\Omega\subset\mathbb{R}^d$, we consider the $\verb|ReLU|_k$ activation function $\sigma(x)=\max(x,0)^k$ with $k\geq1$ and the dictionary \[
\mathbb{P}_k^d=\big\{\sigma(\omega\cdot x+b): \omega\in\mathbb{S}_1^d,~b\in[\underline{b},\bar{b}]\big\},\]
where $\mathbb{S}_1^d$ is the $d$-dimensional unit sphere centered at 0, and $\underline{b}$, $\bar{b}$ are constants depending on $\Omega$. We set $K={\rm co}(\mathbb{P}_k^d)$ and $X=L_2(\Omega)$. It has been shown in Ref.~\citen{JonathanXuFoCM} that 
\begin{subequations}\label{Pkd}
    \begin{align}
        d_n({\rm co}(\mathbb{P}_k^d))_{L_2(\Omega)}&=\mathcal{O}(n^{-\frac{2k+1}{2d}}),\\ \varepsilon_n({\rm co}(\mathbb{P}_k^d))_{L_2(\Omega)}&=\mathcal{O}(n^{-\frac{1}{2}-\frac{2k+1}{2d}}).
    \end{align}
\end{subequations}
Combining \eqref{Pkd} and Corollary \ref{sigmaepsiloncor} we arrive at $$d_n({\rm co}(\mathbb{P}_k^d))_{L_2(\Omega)}\leq\sigma_n({\rm co}(\mathbb{P}_k^d))_{L_2(\Omega)}\lesssim\sqrt{n}\varepsilon_n({\rm co}(\mathbb{P}_k^d))_{L_2(\Omega)},$$ and obtain the convergence rate $$\sigma_n({\rm co}(\mathbb{P}_k^d))_{L_2(\Omega)}=\mathcal{O}(n^{-\frac{2k+1}{2d}}).$$ 
As a result, the error estimate in Corollary \ref{sigmaepsiloncor} is not improvable. 
\end{proof}


In the following we discuss the advantage of our direct comparisons over the classical ones.
The first direct comparison result for greedy-type RBMs is derived in Ref.~\citen{BuffaMaday2012}:
\begin{equation*}
    \sigma_n(K)\lesssim n2^nd_n(K),
\end{equation*}
which is useful if $d_n(K)=o(n^{-1}2^{-n})$. Later on,
this estimate has been improved in Refs.~\citen{BinevCohenDahmenDeVore2011} and \citen{DeVorePetrova2013}. Currently the sharpest direct comparison  in the RBM literature (see Ref.~\citen{DeVorePetrova2013}) is 
\begin{equation}\label{sigma2n}
    \sigma_{2n}(K)\lesssim\sqrt{d_n(K)}.
\end{equation}
Next we show that our direct comparison result in Theorem \ref{sigmaepsilonthm} is indeed sharper than \eqref{sigma2n} as $n\to\infty.$ In doing so, we need the well-known Carl's inequality (see Refs.~\citen{Carl1981} and \citen{Lorentz1996}).
\begin{lemma}[Carl's inequality]\label{Carlthm}
 For every $\alpha>0$, there exists a {constant} $C(\alpha)>0$ such that for any compact set $K$ in a Banach space $X$,
\begin{equation*}
    \varepsilon_n(K)\leq C(\alpha)n^{-\alpha}\max_{1\leq i\leq n}(i^\alpha d_{i-1}(K)).
\end{equation*}
\end{lemma}
For a compact set $K$ with $d_n(K)\lesssim n^{-s}$ and $s>0$,
Lemma \ref{Carlthm} implies that $\varepsilon_n(K)\lesssim n^{-s}$, i.e., the entropy number decays as fast as the polynomial-decaying Kolmogorov $n$-width.
Therefore combining the fact  (see Ref.~\citen{Lorentz1996})  
\begin{equation}\label{dnKcoK}
   d_n(K)=d_n({\rm co}(K)) 
\end{equation}
with Lemma \ref{Carlthm} and Corollary \ref{sigmaepsiloncor}, we have 
\begin{equation}
    d_n(K)\lesssim n^{-s}\Longrightarrow\sigma_n(K)\lesssim \sqrt{n}\varepsilon_n({\rm co}(K))\lesssim n^{\frac{1}{2}-s}.
\end{equation}
In contrast, the estimate \eqref{sigma2n} under the same assumption yields
\begin{equation*}
    d_n(K)\lesssim n^{-s}\Longrightarrow\sigma_n(K)\lesssim n^{-\frac{s}{2}}.
\end{equation*}
Thus our direct comparisons in Theorem \ref{sigmaepsilonthm} are asymptotically sharper than the classical result \eqref{sigma2n} when $s>1$.

\subsection{Examples}\label{subsecExamples}
In the rest of this section, we  discuss  applications of our results to the {second} order elliptic equation 
\begin{subequations}\label{Poisson}
    \begin{align}
    -\nabla\cdot(a\nabla u(a))&=f\quad\text{ in }\Omega\subset\mathbb{R}^d,\\
    u(a)&=0\quad\text{ on }\partial\Omega.
\end{align}
\end{subequations}
Here $f\in H^{-1}(\Omega)$ is fixed, the diffusion coefficient $a\in L_\infty(\Omega)$ is the varying input parameter belonging to the function class $\mathcal{A}$ that will be specified later, and $u(a)\in H_0^1(\Omega)$ is the weak solution corresponding to $a$. First we present a simple lemma for calculating the entropy numbers of the solution manifold $\mathcal{M}=u(\mathcal{A})$ of \eqref{Poisson}.
\begin{lemma}\label{entropylemma}
    Let $0<\alpha\leq1$ and $\Phi: X\rightarrow Y$ be an $\alpha$-H\"older continuous mapping between Banach spaces $X$ and $Y$ such that for all $x_1, x_2\in X,$
    \begin{equation*}
        \|\Phi(x_1)-\Phi(x_2)\|_Y\leq L\|x_1-x_2\|^\alpha_X.
    \end{equation*}
    Then for any compact set $K\subset X$ we have
    \begin{equation*}
        \varepsilon_n(\Phi(K))_Y\leq  L\varepsilon_n(K)^\alpha_X.    \end{equation*}
\end{lemma}
\begin{proof}
    Let $K$ be covered by $2^n$ balls $\{B_\varepsilon(x_i)\}_{i=1}^{2^n}$ of radius $\varepsilon=\varepsilon_n(K)_X$, where $B_\varepsilon(x_i)$ is  centered at $x_i\in X$. Then the H\"older continuity of $\Phi$ implies that $\Phi(K)$ can be covered by balls $\{B_{L\varepsilon^\alpha}(\Phi(x_i))\}_{i=1}^{2^n}$, suggesting that $\varepsilon_n(\Phi(K))_Y\leq L\varepsilon^\alpha.$
\end{proof}

For parametric PDEs, it is generally much more difficult to calculate the $n$-width $d_n(u(\mathcal{A}))$ of the solution manifold $\mathcal{M}=u(\mathcal{A})$ than computing $\varepsilon_n(u(\mathcal{A}))$. To remedy this situation, Cohen and DeVore \cite{CohenDeVore2015,CohenDeVore2016} developed a technical tool for estimating  $d_n(u(\mathcal{A}))$ in terms of $d_n(\mathcal{A})$. 
\begin{theorem}[Theorem 1 from Ref.~\citen{CohenDeVore2016}]\label{CohenDeVore}
 For a pair of complex Banach spaces $X$ and $Y$, suppose $u$ is a holomorphic mapping from an open set $O\subset X$ into $Y$ and $u$ is uniformly bounded on $O$:
\begin{equation*}
\sup_{a\in O}\|u(a)\|_Y<\infty.
\end{equation*}
If $K\subset O$ is a compact set, then for any $\alpha>1$ and $\beta<\alpha-1$ we have
\begin{equation*}
    \sup_{n\geq1} n^\alpha d_n(K)_X<\infty\Longrightarrow\sup_{n\geq1} n^\beta d_n(u(K))_Y<\infty~.
\end{equation*}
\end{theorem}

\subsubsection{Example 1}
Let $C^s(\Omega)=C^{k,\alpha}(\Omega)$ with regularity index $$s:=k+\alpha$$ be the space of functions  on $\Omega$ whose $k$-th derivatives are H\"older continuous with exponent $\alpha\in(0,1]$. In the first example, we focus on the following function class of diffusion coefficients
\begin{equation*}
    \mathcal{A}=\left\{a\in L_\infty(\Omega): \inf_{x\in\Omega}a(x)>M_0>0,~\|a\|_{C^{k,\alpha}(\Omega)}\leq M_1\right\},
\end{equation*} 
where constants $M_0$, $M_1$ are fixed.
In practice, $a\in\mathcal{A}$ has additional affinely parametrized structures allowing an efficient online implementation in RBMs, see, e.g., Refs.~\citen{HesthavenRozzaStamm2016} and \citen{Quarteroni2016}. Otherwise, researchers often utilize the empirical interpolation method \cite{BarraultMadayNguyen2004} to construct affinely parametrized approximations to $a$ and then apply RBMs to the approximate model.

Using Theorem \ref{CohenDeVore}, the holomorphy of the solution operator $u: a\mapsto u(a)$, and $d_n(\mathcal{A})_{L_\infty(\Omega)}=\mathcal{O}(n^{-\frac{s}{d}})$ (see Refs.~\citen{Lorentz1962} and \citen{CohenDeVore2016}), it has been proved in Ref.~\citen{CohenDeVore2016} that
\begin{equation}\label{dnuACohenDeVore}
d_n(u(\mathcal{A}))_{H^1(\Omega)}\lesssim n^{-\frac{s}{d}+1+t}
\end{equation}
for any $t>0.$ Then as a result of \eqref{dnuACohenDeVore} and the classical convergence estimate of greed-type RBMs (see Ref.~\citen{BinevCohenDahmenDeVore2011}), for $s>d$,
\begin{equation}\label{sigmanuACohenDeVore}
\sigma_n(u(\mathcal{A}))_{H^1(\Omega)}\lesssim n^{-\frac{s}{d}+1+t}.
\end{equation}
However, \eqref{sigmanuACohenDeVore} fails to indicate any convergence of RBMs for approximating the solution process $a\mapsto u(a)$ of \eqref{Poisson} when $s\leq d$. 

Our convergence results for RBMs shed some light on  such cases. For \eqref{Poisson} there exists a useful perturbation result (see Theorem 2.1 from Ref.~\citen{BonitoDeVoreNochetto2013})
\begin{equation}\label{purturbation}
\|u(a_1)-u(a_2)\|_{H^1(\Omega)}\leq C_\Omega\|\nabla u(a)\|_{L_p(\Omega)}\|a_1-a_2\|_{L_q(\Omega)},    
\end{equation}
where $p\geq2$ and $q:=2p/(p-2)$. In particular, \eqref{purturbation} with $p=2$
implies the Lipschitz property of
\[
u:L_\infty(\Omega)\rightarrow H^1(\Omega),\quad a\mapsto u(a).
\]
Combining this fact with Lemma \ref{entropylemma} and $\varepsilon_n(\mathcal{A})_{L_\infty(\Omega)}=\mathcal{O}(n^{-\frac{s}{d}})$ (see Ref.~\citen{KolmogorovTihomirov1961}), we have 
\begin{equation}\label{epsilonuA}
    \varepsilon_n(u(\mathcal{A}))_{H^1(\Omega)}\lesssim\varepsilon_n(\mathcal{A})_{L_\infty(\Omega)}\lesssim n^{-\frac{s}{d}}.
\end{equation}
Next we need a lemma from Ref.~\citen{CarlKyreziPajor1999} about the relations between the entropy numbers of $K$ and ${\rm co}(K)$, see also Refs.~\citen{Steinwart2000} and \citen{Gao2001} similar results. 
\begin{lemma}[Proposition 6.2 from Ref.~\citen{CarlKyreziPajor1999}]\label{entropyKcoK}
    Let $K$ be a compact set in a Banach space $X$ of type $p\in(1,2]$, and assume that
\begin{equation*}
\varepsilon_n(K)_X\lesssim  n^{-\alpha}    
\end{equation*}
for some $\alpha>1-1/p$. Then there exists a constant $c(\alpha,p)$ such that
\begin{equation*}
    \varepsilon_n({\rm co}(K))_X\leq c(\alpha,p)n^{-1+\frac{1}{p}}(\log(n+1))^{-\alpha+1-\frac{1}{p}}.
\end{equation*}
\end{lemma}
In particular, the Hilbert space $X=H^1(\Omega)$ is of type $p=2$. As a result of Lemma \ref{entropyKcoK} and \eqref{epsilonuA}, when $d/2<s\leq d$ we have
\begin{equation}\label{epsilonncouA}
    \varepsilon_n({\rm co}(u(\mathcal{A})))_{H^1(\Omega)}\lesssim n^{-\frac{1}{2}}(\log(n+1))^{-\frac{s}{d}+\frac{1}{2}}.
\end{equation}
Therefore using Corollary \ref{sigmaepsiloncor} and \eqref{epsilonncouA} we obtain the convergence of the weak greedy algorithm for \eqref{Poisson}:
\begin{equation*}
d_n(u(\mathcal{A}))_{H^1(\Omega)}\leq\sigma_n(u(\mathcal{A}))_{H^1(\Omega)}\lesssim(\log(n+1))^{-\frac{s}{d}+\frac{1}{2}},
\end{equation*}
where the function class $\mathcal{A}$ is of lower regularity $s\in(d/2,d]$.

\subsubsection{Example 2}
In the second example, we consider the same boundary value problem \eqref{Poisson} on $\Omega=[0,1]^2$ with a different set of diffusion coefficients, which corresponds to DeVore's geometric model in Ref.~\citen{DeVore2014}. Let $$\mathcal{H}=\big\{\phi\in C^{k,\alpha}[0,1]: 0\leq\phi(x)\leq1\text{ for }x\in[0,1]\big\}.$$ 
The graph of a function $\phi\in\mathcal{H}$ partitions $\Omega$ into subregions $\Omega^\phi_+$ and $\Omega^\phi_-$, defined by
\begin{equation}
    \Omega^\phi_+ = \{(x,y)\in \Omega,~y \geq \phi(x)\}
\end{equation}
and $\Omega^\phi_- = \Omega\backslash\Omega^\phi_+$.
The diffusion coefficient $a$ is a member of 
\begin{equation}
    \mathcal{A}=\{a=a_\phi\in L_\infty(\Omega): a|_{\Omega^\phi_+}=2,~ a|_{\Omega^\phi_-}=1\text{ for }\phi\in\mathcal{H}\}.
\end{equation}
Functions in $\mathcal{A}$ have discontinuities along the curve represented by $\phi\in\mathcal{H}$, posing a real challenge for reducing the model problem \eqref{Poisson}. In fact, there is no convergence result of the RBM for the geometric model in the classical literature.

Elementary arguments show that for $q\geq1$ the mapping $\phi\mapsto a_\phi$ is $1/q$-H\"older continuous from $L_\infty[0,1]$ to $L_q(\Omega)$. Therefore using Lemma \ref{entropylemma} we can estimate the metric entropy of $\mathcal{A}$ as
\begin{equation}\label{epsilonC}
    \varepsilon_n(\mathcal{A})_{L_q(\Omega)}\leq \varepsilon_n(\mathcal{H})_{L_\infty[0,1]}^\frac{1}{q}\eqsim n^{-\frac{s}{q}}.
\end{equation}
Moreover, there exists an exponent $P>2$ depending on $\Omega$ such that $\nabla u(a)\in L_P(\Omega)$ (see Refs.~\citen{Grisvard1992} and \citen{BonitoDeVoreNochetto2013}). Then the perturbation result \eqref{purturbation} with $p=P$, $q=Q:=2P/(P-2)\in(2,\infty)$ implies that
$$u:L_Q(\Omega)\rightarrow H^1(\Omega),\quad a\mapsto u(a)$$
is a Lipschitz mapping. Then
using \eqref{epsilonC}
and Lemma \ref{entropylemma} again, we arrive at 
\begin{equation*}
\varepsilon_n(u(\mathcal{A}))_{H^1(\Omega)}\lesssim \varepsilon_n(\mathcal{A})_{L_Q(\Omega)}\lesssim n^{-\frac{s}{Q}}.
\end{equation*}
Combining it with Corollary \ref{sigmaepsiloncor} and Lemma \ref{entropyKcoK} shows that
\begin{equation*}
    \sigma_n(u(\mathcal{A}))_{H^1(\Omega)}\lesssim\log( n+1)^{-\frac{s}{Q}+\frac{1}{2}},\quad s>Q/2,
\end{equation*}
indicating the convergence of the RBM Algorithm \ref{WGA} with $K=u(\mathcal{A})$. In addition, we obtain the novel decay result of the $n$-width $$d_n(u(\mathcal{A}))_{H^1(\Omega)}\lesssim \log( n+1)^{-\frac{s}{Q}+\frac{1}{2}}.$$

\section{Convergence of RBMs in Banach spaces}\label{secBanach}
In this section, we derive direct comparison and convergence estimates for the weak greedy algorithm {(Algorithm \ref{WGA})}, where $X$ is only assumed to be a general Banach space. First we introduce the Banach-Mazur distance $d(X,Y)$ (see Ref.~\citen{Wojtaszczyk2015}) between two isomorphic Banach spaces $X$ and $Y$:
\begin{equation*}
    d(X,Y)=\inf\big\{\|T\|\|T^{-1}\|: T\text{ is an isomorphism from }X \text{ to }Y \big\}.
\end{equation*}
Let $\ell_2^n$ denote the space $\mathbb{R}^n$ under the Euclidean $\ell_2$-norm. For all $n\geq1$ we consider the number
\begin{equation}\label{gamman}
    {\delta}_n(X):=\sup_{Y_n\subset X}d(Y_n,\ell_2^n)\leq\sqrt{n},
\end{equation}
where the supremum is taken over all $n$-dimensional Banach subspaces $Y_n$ in $X$, and the upper bound in \eqref{gamman} has been {proven} in {Section} III.B.9 of Ref.~\citen{Wojtaszscyk1991}. The quantity ${\delta}_n(X)$ measures how far an $n$-dimensional subspace of $X$ is away from the Hilbert space $\ell_2^n$. With the help of \eqref{gamman}, we are able to establish the key lemma for our analysis in Banach spaces.
\begin{lemma}\label{mainlemmaBanach}
Let $K$ be a compact set in a Banach space $X$. For any $v_1, \ldots, v_n\in K$ with $X_k={\rm span}\{v_1,\ldots,v_k\}$ and $X_0=\{0\}$, we have
\begin{equation*}
    \Big(\prod_{k=1}^n\operatorname{dist}(v_k,X_{k-1})\Big)^\frac{1}{n}\leq{\delta}_n(X)(n!V_n)^{\frac{1}{n}}\varepsilon_n({\rm co}(K))_X.
\end{equation*}
\end{lemma}
\begin{proof}
Let $T$ be an isomorphism from $X_n$ to $\ell_2^n$ and $$\|v\|_T:=\|Tv\|_{\ell_2^n}\text{ for any }v\in X_n.$$
Then we have the norm equivalence
\begin{equation}\label{normequiv}
    \|T^{-1}\|^{-1}\|v\|_X\leq\|v\|_T\leq\|T\|\|v\|_X,\quad\forall v\in X_n.
\end{equation}
In addition,
by the definition of $\delta_n(X)$ we can choose $T$ such that
\begin{equation}\label{TTinverse}
    \|T\|\|T^{-1}\|={\delta}_n(X).
\end{equation}
Let $X_{n,T}$ denote the Hilbert space, which is the same as $X_n$ as a set, and is equipped with the inner product $$(\bullet,\bullet)_T:=(T\bullet,T\bullet)_{\ell_2^n}.$$  
As a result of \eqref{normequiv} it holds that
\begin{equation}\label{distk}
    \operatorname{dist}(v_k,X_{k-1})\leq\|T^{-1}\|\|v_k-Q_{k-1}v_k\|_T,
\end{equation}
where $Q_k$ is the orthogonal projection onto $X_k$ with respect to $(\bullet,\bullet)_T$. Let $S={\rm co}\big\{v_1,\ldots,v_n\big\}.$
Then applying our previous estimate in Lemma \ref{mainlemma} to the Hilbert space $X_{n,T}$ and using \eqref{distk}, we have 
\begin{equation}\label{Banach1}
    \begin{aligned}
        \left(\prod_{k=1}^n\operatorname{dist}(v_k,X_{k-1})\right)^\frac{1}{n}&\leq\|T^{-1}\|\left(\prod_{k=1}^n\|v_k-Q_{k-1}v_k\|_T\right)^\frac{1}{n}\\
        &\leq\|T^{-1}\|(n!V_n)^\frac{1}{n}\varepsilon_n(S)_{X_{n,T}}.
    \end{aligned}
\end{equation}
On the other hand, the lower bound in \eqref{normequiv} leads to 
\begin{equation}\label{Banach2}
    \varepsilon_n(S)_{X_{n,T}}\leq\|T\|\varepsilon_n(S)_X\leq\|T\|\varepsilon_n({\rm co}(K))_X.
\end{equation}
Therefore combining \eqref{Banach1} with \eqref{Banach2} and \eqref{TTinverse} completes the proof.
\end{proof}
The estimate \eqref{gamman} on the Banach-Mazur distance can be improved in special Banach spaces like $L_p(d\mu)$ with $1 < p<\infty$. For instance, it has been shown in Ref.~\citen{Wojtaszscyk1991} that
\begin{equation}\label{gammanp}
    {\delta}_n(X)\leq n^{|\frac{1}{2}-\frac{1}{p}|}\quad\text{ for }X=L_p(d\mu),~p\in[1,\infty].
\end{equation}

Following the same proof of Theorem \ref{sigmaepsilonthm}, we  obtain the following error bounds of Algorithm \ref{WGA} in Banach spaces by Lemma \ref{mainlemmaBanach}.
\begin{theorem}\label{sigmaepsilonBanachthm}
Let $K$ be a compact set in a Banach space $X$. For the weak greedy algorithm {(Algorithm \ref{WGA})} we have
\begin{equation*}
    \sigma_n(K)_X\leq\gamma^{-1}{\delta}_n(X)(n!V_n)^\frac{1}{n}\varepsilon_n({\rm co}(K))_X.
\end{equation*}
In particular, the asymptotic error estimates hold:
\begin{equation*}
    \sigma_n(K)_X\lesssim\left\{\begin{aligned}
    &\gamma^{-1}n\varepsilon_n({\rm co}(K))_X,\quad\text{ for a general Banach space }X,\\
    &\gamma^{-1}n^{\frac{1}{2}+|\frac{1}{2}-\frac{1}{p}|}\varepsilon_n({\rm co}(K))_X,\quad X=L_p(d\mu),~p\in[1,\infty].
\end{aligned}\right.
\end{equation*}
\end{theorem}

As a byproduct of \eqref{dnsigman} and Theorem \ref{sigmaepsilonBanachthm}, we also have a direct comparison between the Kolmogorov $n$-width and entropy number in general Banach spaces, which generalizes Proposition 2 from Ref.~\citen{JonathanXuFoCM} for Hilbert spaces.
\begin{corollary}
For a compact set $K$ in a Banach space $X$ we have
\begin{equation*}
    d_n(K)_X\leq{\delta}_n(X)(n!V_n)^\frac{1}{n}\varepsilon_n({\rm co}(K))_X,\quad n\geq1.
\end{equation*}
\end{corollary}

In the end of this section, we shall show that the factor $n^{\frac{1}{2}+|\frac{1}{2}-\frac{1}{p}|}$ in Theorem \ref{sigmaepsilonBanachthm} cannot be removed when $p\geq2$. To that end, we need the following lemma concerning the metric entropy of the $\ell_1^n$ unit ball in $\ell_p$. We remark that we do not know whether Theorem \ref{sigmaepsilonBanachthm} is tight when $1\leq p < 2$.
\begin{lemma}[Theorem 1 from Ref.~\citen{Schutt1984}]\label{Schuttlemma}
    Let $B^1_m$ be an $\ell_1$-unit ball in $\mathbb{R}^m$. For $m/2<s<m$ and $1\leq p\leq\infty$ it holds that
    \begin{equation*}
    \varepsilon_s(B^1_m)_{\ell_p}\eqsim \log(m/s) s^{-1+\frac{1}{p}}.
    \end{equation*}
\end{lemma}
\begin{proposition}
There exists a symmetric and convex compact set $K$ in the Banach space $X=\ell_p$ with $2\leq p\leq\infty$ such that Algorithm \ref{GA} satisfies
\begin{equation}\label{sigmaepsilonequivalence}
    \sigma_n(K)_X\eqsim n^{1-\frac{1}{p}}\varepsilon_n(K)_X.
\end{equation}
\end{proposition}
\begin{proof}
Let $\{x_i\}_{i=1}^\infty$ be a sequence of real numbers decreasing to 0. We consider the compact set 
$$K={\rm co}\big(\cup_{i=1}^\infty\{f_i\}\big)\subset X=\ell_p$$ where $f_i=x_ie_i$ and $e_i$ is the $i$-th unit vector in $\ell_p.$ Clearly the greedy algorithm \ref{GA} selects $f_1, f_2, \ldots, f_n, \ldots$ as the iterates and
\begin{equation}\label{sigmanK}
    \sigma_n(K)=\|f_n\|_{\ell_p}=|x_n|.
\end{equation}
Given an exponent $\alpha>1$, let $x_j=2^{-k\alpha}$ for $2^{k-1}\leq j\leq 2^k-1$. By construction $x_n=\mathcal{O}(n^{-\alpha})$ and we shall  show that \begin{equation}\label{epsilonN}
    \varepsilon_N(K)_{\ell_p}=\mathcal{O}(N^{-1+\frac{1}{p}-\alpha}) \text{ for }N=3\cdot2^{n}.
\end{equation} 
We consider the following subsets of $K$: 
\begin{align*}
K_0&={\rm co}\big(\cup_{i=1}^{2^n}\{f_i\}\big),\\
K_k&={\rm co}\big(\cup_{i=2^{n+k-1}}^{2^{n+k}-1}\{f_i\}\big),\quad 1\leq k\leq n,\\
K_\infty&={\rm co}\big(\cup_{i=2^{2n}}^{\infty}\{f_i\}\big).
\end{align*}
Using the elementary fact
$\varepsilon_{s+t}(A+B)\leq\varepsilon_s(A)+\varepsilon_t(B)$ (see~Refs.~\citen{Lorentz1996} and \citen{JonathanXuFoCM}) we have
\begin{equation}\label{epsilonsplit}
    \varepsilon_{N}(K)_{\ell_p}\leq\varepsilon_{2^n}(K_0)_{\ell_p}+\sum_{k=1}^n\varepsilon_{2^{n-k}}(K_k)_{\ell_p}+\varepsilon_{2^n}(K_\infty)_{\ell_p}.
\end{equation}
Using the formula for the volume of an $\ell_p$ ball of radius $\varepsilon$ in $\mathbb{R}^m$ (see Refs.~\citen{Dirichlet1839} and \citen{Wang2005})
\begin{equation*}
    V^p_m(\varepsilon)=2^m\varepsilon^m\frac{\Gamma(1+\frac{1}{p})^m}{\Gamma(1+\frac{m}{p})},
\end{equation*}
we estimate $\varepsilon_{m}(K_0)_{\ell_p}$ by
calculating the volume of $K_0$, with $m=2^n$:
\begin{equation}\label{K0}
    \begin{aligned}
        \varepsilon_{m}(K_0)_{\ell_p}\eqsim\left(\frac{2^m}{m!}|x_1|\cdots|x_m|\frac{\Gamma(1+\frac{m}{p})}{2^m\Gamma(1+\frac{1}{p})^m}\right)^{\frac{1}{m}}\eqsim m^{-1+\frac{1}{p}-\alpha}.
    \end{aligned}
\end{equation}
For $1\leq k\leq n$, Lemma \ref{Schuttlemma} implies  that
    \begin{equation*}
    \begin{aligned}
        \varepsilon_{2^{n-k}}(K_k)_{\ell_p}&\lesssim2^{-(n+k)\alpha}\varepsilon_{2^{n-k}}\big(B^1_{2^{n+k}}\big)_{\ell_p}\\
        &\lesssim2^{-(n+k)\alpha}\sqrt{k} 2^{(-1+\frac{1}{p})(n-k)}\\
        &\lesssim\sqrt{k}2^{-k(\alpha-1)}2^{-n(\alpha+1-\frac{1}{p})}.
    \end{aligned}
    \end{equation*}
As a result, we obtain the following bound
\begin{equation}\label{Kk}
\sum_{k=1}^n\varepsilon_{2^{n-k}}(K_k)_{\ell_p}\lesssim\sum_{k=1}^\infty\left(\sqrt{k}2^{-k(\alpha-1)}\right)2^{-n(\alpha+1-\frac{1}{p})}\lesssim2^{-n(\alpha+1-\frac{1}{p})}.
    \end{equation}
Since the $\ell_p$ distance from 0 to the elements of $K_\infty$ is less than $2^{-2n\alpha}$, one can simply estimate $\varepsilon_{2^n}(K_\infty)_{\ell_p}$ by 
\begin{equation}\label{Kinf}
    \varepsilon_{2^n}(K_\infty)_{\ell_p}\lesssim 2^{-2n\alpha}.
\end{equation}
Therefore combining \eqref{K0}, \eqref{Kk} and \eqref{Kinf}, we confirm \eqref{epsilonN} and thus \begin{equation}\label{epsilonalphap1}
    \varepsilon_n(K)_{\ell_p}\lesssim n^{-\alpha-1+\frac{1}{p}}\quad\text{ for }n\geq1.
\end{equation} 
Finally \eqref{sigmaepsilonequivalence} with $p\geq2$ follows from Theorem \ref{sigmaepsilonBanachthm} and \eqref{sigmanK}, \eqref{epsilonalphap1}. 
\end{proof}

\section{Convergence analysis of OGAs}\label{secOGA}
Let $X$ be a real Hilbert space and
$\mathbb{D}\subset X$ be a fixed dictionary or collection of functions.
To analyze the convergence behavior of the OGA, we consider the following norm and space \cite{DeVoreTemlyakov1996,BarronCohenDahmenDeVore2008}:
\begin{align*}
    \|f\|_{\mathcal{L}_1(\mathbb{D})}&:=\inf\left\{\sum_i|c_i|: f=\sum_{g_i\in\mathbb{D}}c_ig_i\right\},\\
    \mathcal{L}_1(\mathbb{D})&:=\big\{f\in X: \|f\|_{\mathcal{L}_1(\mathbb{D})}<\infty\big\}.
\end{align*}
Here $\|f\|_{\mathcal{L}_1(\mathbb{D})}$ is sometimes known as the variation of $f$ with respect to $\mathbb{D}$ in the classical literature (see Ref.~\citen{KurkovSanguineti2001}). For the OGA \ref{OGA} with $f\in\mathcal{L}_1(\mathbb{D})$ and {$\|\mathbb{D}\|:=\sup_{g\in\mathbb{D}}\|g\|\leq1$}, the classical error estimate (see Refs.~\citen{DeVoreTemlyakov1996} and \citen{BarronCohenDahmenDeVore2008}) reads 
\begin{equation}\label{OGArateclassical}
    \|f-f_n\|\leq \|f\|_{\mathcal{L}_1(\mathbb{D})}(n+1)^{-\frac{1}{2}}.
\end{equation}
Recently, the work \cite{JonathanXu2022} developed a rate-optimal convergence estimate of the OGA based on the entropy numbers of ${\rm co}(\mathbb{D})$. For example, convergence of the OGA based on a sufficiently regular dictionary like $\mathbb{P}_k^d$ would be much faster than $\mathcal{O}(n^{-\frac{1}{2}})$.
\begin{theorem}[Theorem 1 from Ref.~\citen{JonathanXu2022}]\label{rateOGAthm}
Let $f\in\mathcal{L}_1(\mathbb{D})$ and $t>0$. For Algorithm \ref{OGA} it holds that
\begin{equation*}
    \varepsilon_n({\rm co}(\mathbb{D}))\lesssim n^{-\frac{1}{2}-t}\Longrightarrow\|f-f_n\|\lesssim \|f\|_{\mathcal{L}_1(\mathbb{D})}n^{-\frac{1}{2}-t}.
\end{equation*}
\end{theorem}

\subsection{Direct comparison estimate}
Although the OGA and the greedy-type RBM are not closely related research areas, we show that the main lemma \ref{mainlemma} used in the analysis of greedy RBMs can be used to derive a new type of clean and transparent \emph{direct} comparison between $\|f-f_n\|$ and $\varepsilon_n({\rm co}(\mathbb{D}))$. It turns out that our analysis greatly simplifies the argument in Ref.~\citen{JonathanXu2022} and leads to an improvement of Theorem \ref{rateOGAthm}. To this end, we need the following upper bound for recursively related sequences. {We note that an equivalent result has previously appeared in a slightly different form in the literature, see Lemma 3.1 in Ref.~\citen{Temlyakov2000} or Lemma 3.4 in Ref.~\citen{DeVoreTemlyakov1996}. For the readers convenience we provide the complete proof here}. 
\begin{lemma}\label{inductionlemma}
Let $\{a_n\}_{n\geq0}$ and $\{b\}_{n\geq1}$ be non-negative sequences satisfying
\begin{equation}\label{inductionrecurrence}
        a_n\leq a_{n-1}(1-b_na_{n-1})\text{ for }n\geq1.
\end{equation}
and $b_0=1/a_0$. Then we have
    \begin{equation*}
        a_n\leq\frac{1}{b_0+b_1+\cdots+b_n}.
    \end{equation*}
\end{lemma}
\begin{proof}
We prove the lemma by induction. First we note $\{a_n\}_{n\geq0}$ is a decreasing sequence. By definition we have $a_0\leq1/b_0$ for $n=0$. For the time being assume \begin{equation}\label{inductionassumption}
    a_{n-1}\leq\frac{1}{b_0+\cdots+b_{n-1}}.
\end{equation} 
If $a_{n-1}\leq\frac{1}{b_0+\cdots+b_{n}}$, then the monotonicity of $\{a_n\}_{n\geq1}$ implies $$a_n\leq a_{n-1}\leq\frac{1}{b_0+\cdots+b_{n}}.$$
If $a_{n-1}>\frac{1}{b_0+\cdots+b_{n}}$, then by \eqref{inductionrecurrence} and the induction assumption \eqref{inductionassumption}, we obtain
\begin{equation*}
\begin{aligned}
        a_n&\leq \frac{1}{b_0+\cdots+b_{n-1}}\left(1-\frac{b_n}{b_0+\cdots+b_{n}}\right)\\
        &=\frac{1}{b_0+\cdots+b_{n}}.
        \end{aligned}
        \end{equation*}
Therefore the induction is complete and Lemma \ref{inductionlemma} is true. 
\end{proof} 
\begin{theorem}\label{OGAthm}
For $f\in\mathcal{L}_1(\mathbb{D})$ and Algorithm \ref{OGA}
we have
\begin{equation}\label{OGAestimate}
    \|f-f_n\|\leq \frac{(n!V_n)^\frac{1}{n}}{\sqrt{n}}\|f\|_{\mathcal{L}_1(\mathbb{D})}\varepsilon_n({\rm co}(\mathbb{D})).
\end{equation}
\end{theorem}
\begin{proof}
Let {$r_n=f-f_n$} in Algorithm \ref{OGA}. Following the analysis in Ref.~\citen{JonathanXu2022} {(see also the analysis in Section 3 of Ref.~\citen{Gao2017})}, we have
\begin{equation}\label{rnnm1}
\begin{aligned}
\|r_n\|^2&\leq\left\|r_{n-1}-\frac{\langle r_{n-1},g_n-P_{n-1}g_n\rangle}{\|g_n-P_{n-1}g_n\|^2}(g_n-P_{n-1}g_n)\right\|^2\\
&=\|r_{n-1}\|^2-\frac{|\langle r_{n-1},g_n-P_{n-1}g_n\rangle|^2}{\|g_n-P_{n-1}g_n\|^2}.
 \end{aligned}
\end{equation}
We remark $\|g_n-P_{n-1}g_n\|\neq0$ otherwise $g_n\in X_{n-1}={\rm span}\{g_1,\ldots,g_{n-1}\}$ and $\arg\max_{g\in\mathbb{D}}|\langle r_{n-1},g\rangle|=|\langle r_{n-1},g_n\rangle|=0$, which implies $r_{n-1}=0$ and termination of the OGA. On the other hand, using $r_{n-1}\perp X_{n-1}$ we have
\begin{equation}\label{rnnm2}
    \begin{aligned}
        \|r_{n-1}\|^2&=\langle f,r_{n-1}\rangle\leq\|f\|_{\mathcal{L}_1(\mathbb{D})}|\langle r_{n-1},g_n\rangle|\\
        &=\|f\|_{\mathcal{L}_1(\mathbb{D})}|\langle r_{n-1},g_n-P_{n-1}g_n\rangle|.
    \end{aligned}
\end{equation}
Therefore combining \eqref{rnnm1} and \eqref{rnnm2} leads to 
\begin{equation}\label{rnnm3}
    \begin{aligned}
        \|r_n\|^2\leq\|r_{n-1}\|^2-\frac{\|r_{n-1}\|^4}{\|f\|_{\mathcal{L}_1(\mathbb{D})}^2\|g_n-P_{n-1}g_n\|^2}.
    \end{aligned}
\end{equation}
With the notation 
\begin{align*}
    &a_n=\|f-f_n\|^2/\|f\|^2_{\mathcal{L}_1(\mathbb{D})},\\
    &b_n=\|g_n-P_{n-1}g_n\|^{-2},
\end{align*}
the inequality \eqref{rnnm3} reduces to the 
recurrence relation
\begin{equation}\label{recurrence}
    a_n\leq a_{n-1}(1-b_na_{n-1}).
\end{equation}
By the definition of $\|f\|_{\mathcal{L}_1(\mathbb{D})}$ we have $a_0\leq1$. It then follows from \eqref{recurrence} and Lemma \ref{inductionlemma} that
\begin{equation}\label{induction}
a_n\leq\frac{1}{1+b_1+\cdots+b_n}.
\end{equation}
Using \eqref{induction}, a mean value inequality and the key Lemma \ref{mainlemma}, we obtain
\begin{equation*}
            \begin{aligned}
        a_n&\leq\frac{1}{1+n(b_1\cdots b_n)^\frac{1}{n}}\\
        &=\frac{1}{1+n\big(\prod_{k=1}^n\|g_k-P_{k-1}g_k\|\big)^{-\frac{2}{n}}}\\
    &\leq\frac{1}{1+n(nV_n)^{-\frac{2}{n}}\varepsilon_n({\rm co}(\mathbb{D}))^{-2}}\\
    &\leq\frac{(nV_n)^\frac{2}{n}}{n}\varepsilon_n({\rm co}(\mathbb{D}))^2.
\end{aligned}
\end{equation*}
The proof is complete.
\end{proof}

We remark that the classical Theorem \ref{rateOGAthm} is valid under an asymptotic assumption while Theorem \ref{OGAthm} is a non-asymptotic and unconditional estimate. In addition, the constant in the upper bound of \eqref{OGAestimate} as $n\to\infty$ behaves like
\[
\lim_{n\to\infty}\frac{(n!V_n)^\frac{1}{n}}{\sqrt{n}}=\sqrt{\frac{2\pi}{\text{e}}}.
\]
Therefore Theorem \ref{OGAthm} yields
\begin{equation*}
    \|f-f_n\|\lesssim\|f\|_{\mathcal{L}_1(\mathbb{D})} \varepsilon_n({\rm co}(\mathbb{D})),
\end{equation*}
which is indeed stronger than Theorem \ref{rateOGAthm}.

\subsection{Best $n$-term approximation}\label{subsecnterm}
{Consider the set of $n$-sparse elements 
$$\Sigma_n(\mathbb{D})=\Big\{\sum_{i=1}^nc_i\tilde{g}_i: c_i\in\mathbb{R},~\tilde{g}_i\in\mathbb{D},~i=1,\ldots,n\Big\}.$$ 
It is natural to compare the error of greedy algorithms with the best approximation error $$E_n(f,\mathbb{D})=\inf_{g\in\Sigma_n(\mathbb{D})}\|f-g\|$$
by $n$-sparse elements. When $\mathbb{D}$ is an orthonormal set in $X$, the equality $\|f-f_n\|=E_n(f,\mathbb{D})$ automatically holds (see Ref.~\citen{Temlyakov2008}). In general, 
such convergence analysis using $E_n(f,\mathbb{D})$ often requires near orthogonality assumption on $\mathbb{D}$.
For example, on a $M$-coherent dictionary $\mathbb{D}$, the following Lebesgue inequality is true (see Section 2.6 of Ref.~\citen{Temlyakov2008}):
\[
\|f-f_{\lfloor n\log n\rfloor}\|\leq c_1E_n(f,\mathbb{D}),
\]
for $n\leq c_2M^{-\frac{2}{3}}$
with $c_1$, $c_2$ being explicit constants.
We also refer to Ref.~\citen{Temlyakov2018} for Lebesgue-type inequality of generalized OGAs in Banach spaces.
On the other hand, the entropy-based estimate \eqref{OGAestimate} applies to any normalized dictionary in Hilbert spaces. 

In addition, we note that Theorem \ref{OGAestimate} yields the following direct comparison between $E_n(f,\mathbb{D})$ and $\varepsilon_n({\rm co}(\mathbb{D}))$:
\begin{subequations}
\begin{align}
&E_n(f,\mathbb{D})\leq\frac{(n!V_n)^\frac{1}{n}}{\sqrt{n}}\|f\|_{\mathcal{L}_1(\mathbb{D})}\varepsilon_n({\rm co}(\mathbb{D})),\\
&E_n(\mathbb{D}):=\sup_{f\in{\rm co}(\mathbb{D})}E_n(f,\mathbb{D})\leq\frac{(n!V_n)^\frac{1}{n}}{\sqrt{n}}\varepsilon_n({\rm co}(\mathbb{D})).\label{EnEpsilonn}
\end{align}
\end{subequations}
When the number $|\mathbb{D}|$ of elements in $\mathbb{D}$ is finite, an indirect comparison between $E_n(\mathbb{D})$ and $\varepsilon_n({\rm co}(\mathbb{D}))$ in the converse direction of \eqref{EnEpsilonn} can be found in Section 7.4 of Ref.~\citen{Temlyakov2018}. In particular, combining Theorem 7.4.3 of Ref.~\citen{Temlyakov2018} with Theorem \ref{OGAestimate}, we obtain the following novel comparison between the convergence rate of $E_n(\mathbb{D})$ and the error of OGAs whose target functions belonging to ${\rm co}(\mathbb{D})$.
\begin{corollary}\label{OGAntermrate}
Assume there exists a constant $s>0$ such that
\begin{equation*}
   E_n(\mathbb{D})\lesssim n^{-s},\quad n\leq|\mathbb{D}|,
\end{equation*}
where the cardinality $|\mathbb{D}|$ is finite.
Then for $f_n$ generated by the OGA (Algorithm \ref{OGA}) with $n\leq|\mathbb{D}|$, we have
\begin{equation*}
   \sup_{f\in{\rm co}(\mathbb{D})}\|f-f_n\|\leq C(s)(\log(2|\mathbb{D}|/n))^sn^{-s},
\end{equation*}
where $C(s)$ is a constant depending only on $s.$
\end{corollary}
We note that the above result is independent of the coherence property of $\mathbb{D}$. Moreover, Corollary \ref{OGAntermrate} is not a Lebesgue-type inequality but is in a similar fashion to the classical $n$-width based convergence estimate \eqref{rateRBMthm} for the RBM. It confirms that the OGA collectively achieves the best $n$-term approximation rate up to a logarithmic factor on finite dictionaries.
}

\subsection{General target functions}
In general, if $f$ is not contained in $\mathcal{L}_1(\mathbb{D})$, one can use the next corollary to estimate the error of the OGA.
\begin{corollary}\label{OGAcor}
    For all $f\in X$ and any $h\in\mathcal{L}_1(\mathbb{D})$, Algorithm \ref{OGA} satisfies
\begin{equation*}
    \|f-f_n\|^2\leq\|f-h\|^2+4\frac{(n!V_n)^\frac{2}{n}}{n}\|h\|^2_{\mathcal{L}_1(\mathbb{D})}\varepsilon_n({\rm co}(\mathbb{D}))^2.
\end{equation*}
\end{corollary}
\begin{proof}
We use the same notation as in the proof of Theorem \ref{OGAthm}. We note \eqref{rnnm1} still holds and
\begin{equation*}
    \begin{aligned}
        \|r_{n-1}\|^2&=\langle h,r_{n-1}\rangle+\langle f-h,r_{n-1}\rangle\\
        &\leq\|h\|_{\mathcal{L}_1(\mathbb{D})}|\langle r_{n-1},g_n-P_{n-1}g_n\rangle|+\|f-h\|\|r_{n-1}\|.
    \end{aligned}
\end{equation*}
Then combining the above inequality and a
mean value inequality yields
\begin{equation}\label{rfh}
        \frac{1}{2}\big(\|r_{n-1}\|^2-\|f-h\|^2\big)\leq\|h\|_{\mathcal{L}_1(\mathbb{D})}|\langle r_{n-1},g_n-P_{n-1}g_n\rangle|.
\end{equation}
Let $c_n=(\|r_{n}\|^2-\|f-h\|^2)/(4\|h\|^2_{\mathcal{L}_1(\mathbb{D})})$. Using \eqref{rnnm1} and \eqref{rfh} we obtain
\begin{equation*}
    c_n\leq c_{n-1}(1-b_nc_{n-1}).
\end{equation*}
The rest of the proof is same as Theorem \ref{OGAthm}.
\end{proof}

For $t>0$ and $g\in X$ the $K$-functional $$K(t,g)=\inf_{h\in\mathcal{L}_1(\mathbb{D})}\left(\|g-h\|_X+t\|h\|_{\mathcal{L}_1(\mathbb{D})}\right)$$
measures how well the element of $X$ can be approximated by $\mathcal{L}_1(\mathbb{D})$ with an approximant of small $\mathcal{L}_1(\mathbb{D})$ norm (see Refs.~\citen{BerghLofstrom1976,DeVoreLorentz1993} and \citen{BrennerScott2008}). For the index pair $(\theta,\infty)$ with $\theta\in(0,1)$, the interpolation norm and space between $X$ and $\mathcal{L}_1(\mathbb{D})$  are 
\begin{align*}
    \|g\|_\theta&=\|g\|_{[X,\mathcal{L}_1(\mathbb{D})]_{\theta,\infty}}:=\sup_{0<t<\infty}t^{-\theta}K(t,g),\\
    X_\theta&=\big\{g\in X: \|g\|_\theta<\infty\big\},
\end{align*}
respectively.
As in Refs.~\citen{BarronCohenDahmenDeVore2008} and \citen{JonathanXu2022}, we also obtain an explicit and improved error estimate of the OGA  in terms of the interpolation norm $\|f\|_\theta$.
\begin{proposition}
    For all $f\in X_\theta$ and $\theta\in(0,1)$ Algorithm \ref{OGA} satisfies
\begin{equation*}
    \|f-f_n\|\leq2^\theta\frac{(n!V_n)^\frac{\theta}{n}}{n^\frac{\theta}{2}}\|f\|_\theta \varepsilon_n({\rm co}(\mathbb{D}))^\theta.
\end{equation*}
\end{proposition}
\begin{proof}
    By Corollary \ref{OGAcor}, we obtain for all $h\in\mathcal{L}_1(\mathbb{D})$
    \begin{equation*}
    \|f-f_n\|\leq\|f-h\|+2\frac{(n!V_n)^\frac{1}{n}}{\sqrt{n}}\|h\|_{\mathcal{L}_1(\mathbb{D})}\varepsilon_n({\rm co}(\mathbb{D})).
\end{equation*}
As a result, setting $t=2\frac{(n!V_n)^\frac{1}{n}}{\sqrt{n}}\varepsilon_n({\rm co}(\mathbb{D}))$, we have
\begin{equation*}
    \|f-f_n\|\leq K(t,f)\leq t^\theta\|f\|_\theta
\end{equation*}
and complete the proof.
\end{proof}

\section{Concluding Remarks}\label{secConclusion}
We have derived direct and optimal convergence estimates of the greedy-type RBM and the OGA based on Kolmogorov's entropy numbers. Constants in our upper bounds are explicit and simple. A future research direction is to extend our analysis to POD-greedy algorithms \cite{HaasdonkOhlberger2008,Haasdonk2013} for time-dependent parametric PDEs, which is a combination of greedy algorithms and POD for temporal compression of snapshots. In addition, the generalization of the OGA in Banach spaces is called the Chebyshev Greedy Algorithm (CGA) \cite{Temlyakov2001,Temlyakov2008,DereventsovTemlyakov2019}. It would also be interesting to investigate if our analysis in Section \ref{secOGA} and Lemma \ref{mainlemmaBanach} can be applied to the CGA.

\section*{Acknowledgements}
The authors would like to thank the two anonymous referees for constructive comments  that improve the quality of the manuscript. Subsection \ref{subsecnterm} is partially due to the bibliographic suggestion \cite{Temlyakov2018} of one referee. 

The work of Li was partially supported by the Fundamental Research Funds for the Central Universities 226-2023-00039. Siegel was supported by the National Science Foundation (DMS-2111387 and CCF-2205004).

\bibliographystyle{ws-m3as}

\end{document}